\documentclass[a4paper,12pt]{amsart}
\makeatletter
\let\@@pmod\pmod
\DeclareRobustCommand{\pmod}{\@ifstar\@pmods\@@pmod}
\def\@pmods#1{\mkern4mu({\operator@font mod}\mkern 6mu#1)}
\makeatother
\usepackage{fullpage,mathrsfs,amssymb,color,verbatim}
\newtheorem{theorem}{Theorem}
\newtheorem{proposition}[theorem]{Proposition}
\newtheorem{lemma}[theorem]{Lemma}

\theoremstyle{remark}
\newtheorem*{remarks}{Remarks}
\numberwithin{theorem}{section}
\numberwithin{equation}{section}
\newcommand{\Z}{\mathbb{Z}}
\newcommand{\Q}{\mathbb{Q}}
\newcommand{\R}{\mathbb{R}}
\newcommand{\C}{\mathbb{C}}

\renewcommand{\H}{\mathbb{H}}
\DeclareMathOperator{\sgn}{sgn}
\newcommand{\Res}[1]{\underset{#1}{\text{Res}}\,}

\newcommand{\Ns}{N^s}
\begin{document}
\title{Quantitative estimates for simple zeros of $L$-functions}
\author{Andrew R. Booker}
\address{School of Mathematics, University of Bristol, Bristol, BS8 1TW, UK}
\email{andrew.booker@bristol.ac.uk}
\author{Micah B. Milinovich}
\address{Department of Mathematics, University of Mississippi,
University, MS 38677 USA}
\email{mbmilino@olemiss.edu}
\author{Nathan Ng}
\address{Department of Mathematics and Computer Science, University of
Lethbridge, Lethbridge, AB Canada T1K 3M4}
\email{nathan.ng@uleth.ca}

\thanks{Research of the first author was supported by EPSRC Grant
\texttt{EP/K034383/1}. Research of the second author was supported
by the NSA Young Investigator Grants \texttt{H98230-15-1-0231} and
\texttt{H98230-16-1-0311}. Research of the third author was supported
by NSERC Discovery Grant (RGPIN- 2015-05972). No data were created in the course of
this study.}

\begin{abstract}
We generalize a method of Conrey and Ghosh \cite{CG88} to prove quantitative 
estimates for simple zeros of modular form $L$-functions of arbitrary 
conductor. 
\end{abstract}

\subjclass[2010]{Primary 11F66, 11F11, 11M41}

\maketitle

\section{Introduction}
Let $f\in S_k(\Gamma_1(N))$ be a classical holomorphic modular form of
weight $k$ and level $N$. Assume that $f$ is \emph{primitive}, meaning
that it is a normalized Hecke eigenform in the new subspace. Then it has
a Fourier expansion of the shape
$$
f(z)=\sum_{n=1}^\infty\lambda_f(n)n^{\frac{k-1}2}e^{2\pi inz},
$$
where the $\lambda_f(n)$ are multiplicative and satisfy the Ramanujan
bound $|\lambda_f(n)|\le d(n)$. Let
$\Lambda_f(s)=\Gamma_\C(s+\tfrac{k-1}2)L_f(s)$
denote the complete $L$-function of $f$, with analytic normalization,
where
$$
\Gamma_\C(s)=2(2\pi)^{-s}\Gamma(s)
\quad\text{and}\quad
L_f(s)=\sum_{n=1}^\infty\frac{\lambda_f(n)}{n^s},
$$
and let
$$
\Ns_f(T)=\#\bigl\{\rho\in\C:\Lambda_f(\rho)=0, \Lambda_f'(\rho)\ne0,
|\Im(\rho)|\le T\bigr\}
$$
be the number of simple zeros of $\Lambda_f(s)$ with imaginary part in
$[-T,T]$.

In \cite{MN14}, the second and third authors showed that if $\Lambda_f(s)$
satisfies the Generalized Riemann Hypothesis, then
$$
\Ns_f(T)\ge T(\log{T})^{-\varepsilon}
$$
for any fixed $\varepsilon>0$ and all sufficiently large $T>0$.
Unconditionally,
when $N=1$ and $k=12$, Conrey and Ghosh \cite{CG88}
showed that
\begin{equation}\label{eq:cgestimate}
\forall\varepsilon>0,
\exists T\ge\varepsilon^{-1}\text{ such that }
\Ns_f(T)\ge T^{\frac16-\varepsilon}.
\end{equation}
Moreover, their proof works more generally for $N=1$ and arbitrary $k$,
provided that $\Ns_f(T)$ is not identically $0$.  In light of the first
author's result \cite{Boo16} that $\Ns_f(T)\to\infty$ as $T\to\infty$,
\eqref{eq:cgestimate} holds for all primitive $f$ of
conductor $1$.

In this paper we aim to prove similar unconditional quantitative estimates
of simple zeros for primitive
forms of arbitrary conductor $N$. However, we encounter some obstacles
that are reminiscent of the well-known difficulty of extending Hecke's
converse theorem to arbitrary conductor, and are not present for $N=1$.
Taking inspiration from Weil's generalization \cite{Wei67} of Hecke's
converse theorem, we consider character twists. For a Dirichlet character
$\chi\pmod*{q}$, let $f\otimes\chi$ denote the unique primitive form such
that $\lambda_{f\otimes\chi}(n)=\lambda_f(n)\chi(n)$ for all $n$ coprime
to $q$.
\begin{theorem}\label{thm:twist}
Let $f\in S_k(\Gamma_1(N))$ be a primitive form.
Then there is a Dirichlet character $\chi$ such that
\eqref{eq:cgestimate} holds with $f\otimes\chi$ in place of $f$.
\end{theorem}

Next, for odd conductors we obtain a weaker but unconditional quantitative
estimate for $\Ns_f(T)$, without the twist. Moreover, we show that there is
a sort of ``Deuring--Heilbronn phenomenon'' at play, so that if $\Ns_f(T)$
is unexpectedly small then we can substantially improve our result for
$\Ns_{f\otimes\chi}(T)$.
\begin{theorem}\label{thm:oddN}
Let $f\in S_k(\Gamma_1(N))$ be a primitive form of odd conductor.
Then
$$
\forall\varepsilon>0,
\exists T\ge\varepsilon^{-1}\text{ such that }
\Ns_f(T)\ge\begin{cases}
\exp((\log T)^{\frac13-\varepsilon})&\text{if $k=1$ or $f$ is a CM form},\\
\log\log\log{T}&\text{otherwise}.
\end{cases}
$$
Further, if $\Ns_f(T)\ll1+T^\varepsilon$ for every $\varepsilon>0$, then
\begin{enumerate}
\item[(i)] there is a Dirichlet character $\chi$ such that,
$\forall\varepsilon>0, \exists T\ge\varepsilon^{-1}$
such that $\Lambda_{f\otimes\chi}(s)$ has
at least $T^{\frac12-\varepsilon}$ simple zeros with real part
$\frac12$ and imaginary part in $[-T,T]$;
\item[(ii)] $\Lambda_f(s)$ has simple zeros with real part arbitrarily close to $1$.
\end{enumerate}
\end{theorem}

\begin{remarks}\
\begin{enumerate}
\item The exponent $\frac16$ in \eqref{eq:cgestimate} is related to the
best known subconvexity estimate for modular form $L$-functions in the $t$
aspect; it can be replaced by any $\delta>0$ such that
$L_f(\frac12+it)\ll_{f,\varepsilon}(1+|t|)^{\frac12-\delta+\varepsilon}$
holds for all primitive forms $f$ and all $\varepsilon>0$. In
\cite{BMN19} we showed that $\delta=\frac16$ is admissible.
Very recent work of Munshi \cite{Mun18} improves this to
$\delta=\frac16+\frac1{1200}$ for forms of level $1$, with
a corresponding improvement to \eqref{eq:cgestimate} in that
case.
\item In Theorem~\ref{thm:twist},
one can take the conductor of $\chi$ to be $1$ or a prime number
bounded by a polynomial function of $N$.
\item The proof of Theorem~\ref{thm:oddN} makes use of the idea
originating with Conrey and Ghosh \cite{CG88} of twisting the
coefficients of $L_f(s)$ by
$(-1)^n$ to prevent the main terms of our estimate from cancelling out.
This relies implicitly on the fact that there is no primitive
Dirichlet character of conductor $2$, and is the ultimate reason
for our restriction to odd $N$.
\item The improved estimate in Theorem~\ref{thm:oddN} in the Galois and
CM cases arises from Coleman's Vinogradov-type zero-free
region for Hecke $L$-functions \cite{Col90}.
\end{enumerate}
\end{remarks}

\subsection*{Acknowledgements}
We thank the anonymous referee for helpful suggestions and corrections,
and the Banff International Research Station for hosting us for a
Research in Teams Meeting (15rit201). A significant portion of this
project was completed during that week and we appreciated the excellent
working conditions.

\section{Dirichlet series}
In order to establish the existence of simple zeros it is useful
to study not only $L_f(s)$, but some related Dirichlet series and their additive twists.
This is one of the central ideas in \cite{CG88}.
A key role is played by the series
$$
D_f(s)=L_f(s)\frac{d^2}{ds^2}\log L_f(s)=\sum_{n=1}^{\infty}c_f(n)n^{-s},
$$
which has a meromorphic continuation to $\C$ with poles precisely
at the simple zeros of $L_f(s)$ (including the trivial zeros
$s=\frac{1-k}2-n$ for $n=0, 1, 2, \ldots$).

For $\alpha\in\Q^\times$ and $\chi$ a Dirichlet character, let
$$
L_f(s,\alpha)=\sum_{n=1}^\infty\lambda_f(n)e(\alpha n)n^{-s}
\quad\text{and}\quad
L_f(s,\chi)=\sum_{n=1}^\infty\lambda_f(n)\chi(n)n^{-s}.
$$
Likewise, define
$$
D_f(s,\alpha)=\sum_{n=1}^\infty c_f(n)e(\alpha n)n^{-s}
\quad\text{and}\quad
D_f(s,\chi)=\sum_{n=1}^\infty c_f(n)\chi(n)n^{-s}.
$$

Let $\xi$ denote the nebentypus character of $f$.
Set
$$
Q(N)=\{1\}\cup\{q\text{ prime}:q\nmid N\},
$$
and for each $q\in Q(N)$, define the rational functions
$$
P_{f,q}(x)=\begin{cases}
1&\text{if }q=1,\\
1-\lambda_f(q)x+\xi(q)x^2&\text{otherwise}
\end{cases}
$$
and
$$
R_{f,q}(x)=\begin{cases}
0&\text{if }q=1,\\
\frac{q\log^2{q}}{q-1}
\frac{x(\lambda_f(q)-4\xi(q)x+\lambda_f(q)\xi(q)x^2)}
{P_{f,q}(x)}
&\text{if }q\ne1.
\end{cases}
$$
These are such that, if
$$
\chi_0(n)=\begin{cases}
1&\text{if }(n,q)=1,\\
0&\text{otherwise}
\end{cases}
$$
denotes the trivial character mod $q$, then
$$
L_f(s,\chi_0)=P_{f,q}(q^{-s})L_f(s)
$$
and
\begin{equation}\label{Dfchi0}
D_f(s,\chi_0)=P_{f,q}(q^{-s})D_f(s)-\frac{q-1}{q}R_{f,q}(q^{-s})L_f(s).
\end{equation}

For any $a\in\Z$ coprime to $q$, we define
\begin{align*}
D_{f,a,q}(s)&=D_f(s,\tfrac{a}q)-R_{f,q}(q^{-s})L_f(s)
=\sum_{n=1}^{\infty}c_{f,a,q}(n)n^{-s},\\
D_{f,a,q}^*(s)&=D_{f,a,q}(s)+\psi'(s+\tfrac{k-1}2)L_f(s,\tfrac{a}{q}),
\quad\text{where }\psi(s)=\frac{\Gamma'}{\Gamma}(s)
\end{align*}
and
$$
D_{f,a,q}(s,\alpha)=\sum_{n=1}^{\infty}c_{f,a,q}(n)e(\alpha n)n^{-s}
\quad\text{for }\alpha\in\Q^\times.
$$

To each of $L_f$, $D_f$,
$D_{f,a,q}$, $D_{f,a,q}^*$ and their twists, we define
completed versions $\Lambda_f$, $\Delta_f$, $\Delta_{f,a,q}$,
$\Delta_{f,a,q}^*$ obtained by multiplying by $\Gamma_\C(s+\frac{k-1}2)$.
By the Ramanujan bound $|\lambda_f(q)|\le 2$ and
\cite[Proposition~3.1]{BK11},
$\Delta_f(s,a/q)-\Delta_{f,a,q}^*(s)$ is holomorphic for
$\Re(s)>0$. In turn, the analytic properties of $\Delta_{f,a,q}^*(s)$
are described by the following proposition.
\begin{proposition}\label{voronoi}
Let $f\in S_k(\Gamma_0(N),\xi)$ be a primitive form, $q\in Q(N)$,
and $a\in\Z$ coprime to $q$. Then
$\Delta_{f,a,q}^*(s)$ is a ratio of entire functions of finite order,
has at most simple poles, all
of which are contained in the critical strip
$\{s\in\C:\Re(s)\in(0,1)\}$,
and satisfies the functional equation
\begin{equation}\label{dstarfunceq}
\Delta_{f,a,q}^*(s)=\epsilon\xi(q)(Nq^2)^{\frac12-s}
\Delta_{\bar{f},-\overline{Na},q}^*(1-s),
\end{equation}
where $\bar{f}\in S_k(\Gamma_0(N),\overline{\xi})$ is the dual of $f$,
$\epsilon\in\C^\times$ is the root number of $f$
and $\overline{Na}$ denotes a multiplicative inverse of $Na\pmod*{q}$.
\end{proposition}
\begin{proof}
For $q=1$ the result follows immediately from \cite[(3.1)]{Boo16}, so we
may assume that $q$ is prime.
Let $\chi$ be a Dirichlet character of conductor $q$.
Then the complete twisted $L$-function $\Lambda_f(s,\chi)$
satisfies the functional equation
$$
\Lambda_f(s,\chi)=\epsilon\xi(q)\chi(N)\frac{\tau(\chi)^2}{q}
(Nq^2)^{\frac12-s}\Lambda_{\bar{f}}(1-s,\overline{\chi}),
$$
where $\epsilon\in\C^\times$ is the root number of $f$. Applying
\cite[(3.1)]{Boo16} to $f\otimes\chi$, we thus have
\begin{equation}\label{dfunceq1}
\Delta_f(s,\chi)
-\epsilon\xi(q)\chi(N)\frac{\tau(\chi)^2}{q}
(Nq^2)^{\frac12-s}\Delta_{\bar{f}}(1-s,\overline{\chi})
=\Lambda_f(s,\chi)\bigl(\psi'(\tfrac{k+1}2-s)-\psi'(s+\tfrac{k-1}2)\bigr).
\end{equation}

Next, we have
$$
\Delta_f\!\left(s,\frac{a}q\right)=\Delta_f(s)
-\frac{q}{q-1}\Delta_f(s,\chi_0)
+\frac1{q-1}\sum_{\substack{\chi\pmod*{q}\\\chi\ne\chi_0}}
\tau(\overline{\chi})\chi(a)\Delta_f(s,\chi),
$$
where $\chi_0$ is the trivial character mod $q$.
Combining this with \eqref{Dfchi0} we get
\begin{equation}\label{Deltafaq}
\begin{aligned}
\Delta_{f,a,q}(s)
=\left(1-\frac{q}{q-1}P_{f,q}(q^{-s})\right)\Delta_f(s)
+\frac1{q-1}\sum_{\substack{\chi\;(\text{mod }q)\\\chi\ne\chi_0}}
\tau(\overline{\chi})\chi(a)\Delta_f(s,\chi).
\end{aligned}
\end{equation}
Note in particular that $\Delta_{f,a,q}(s)$ is a ratio of entire
functions of finite order, and all of its poles in $\{s\in\C:\Re(s)>0\}$
are simple and located at simple zeros of either $\Lambda_f(s)$ or
$\Lambda_f(s,\chi)$ for some $\chi\ne\chi_0$.

Note that $P_{f,q}$ satisfies the functional equation
$$
1-\frac{q}{q-1}P_{f,q}(q^{-s})
=\xi(q)q^{1-2s}\left(1-\frac{q}{q-1}P_{\bar{f},q}(q^{s-1})\right),
$$
and thus, by \cite[(3.1)]{Boo16},
\begin{equation}\label{dfunceq2}
\begin{aligned}
\left(1-\frac{q}{q-1}P_{f,q}(q^{-s})\right)&\Delta_f(s)
-\epsilon\xi(q)(Nq^2)^{\frac12-s}
\left(1-\frac{q}{q-1}P_{\bar{f},q}(q^{s-1})\right)
\Delta_{\bar{f}}(1-s)\\
&=\left(1-\frac{q}{q-1}P_{f,q}(q^{-s})\right)
\Lambda_f(s)\bigl(\psi'(\tfrac{k+1}2-s)-\psi'(s+\tfrac{k-1}2)\bigr).
\end{aligned}
\end{equation}
Thus, replacing $f$ by $\bar{f}$, $s$ by $1-s$, $a$ by
$-\overline{Na}$ and $\chi$ by $\overline{\chi}$ in \eqref{Deltafaq},
we get
\begin{align*}
\Delta_{\bar{f},-\overline{Na},q}(1-s)
=\left(1-\frac{q}{q-1}P_{\bar{f},q}(q^{s-1})\right)\Delta_{\bar{f}}(1-s)
+\frac1{q-1}\sum_{\substack{\chi\pmod*{q}\\\chi\ne\chi_0}}
\tau(\chi)\chi(-Na)\Delta_{\bar{f}}(1-s,\overline{\chi}).
\end{align*}
Applying the functional equations \eqref{dfunceq1} and
\eqref{dfunceq2}, together with the relation
$\tau(\chi)\tau(\overline{\chi})=\chi(-1)q$, we thus have
\begin{align*}
&\Delta_{f,a,q}(s)
-\epsilon\xi(q)(Nq^2)^{\frac12-s}
\Delta_{\bar{f},-\overline{Na},q}(1-s)\\
&=\Biggl[
\left(1-\frac{q}{q-1}P_{f,q}(q^{-s})\right)\Lambda_f(s)
+\frac1{q-1}\sum_{\substack{\chi\pmod*{q}\\\chi\ne\chi_0}}
\tau(\overline{\chi})\chi(a)\Lambda_f(s,\chi)\Biggr]
\bigl(\psi'(\tfrac{k+1}2-s)-\psi'(s+\tfrac{k-1}2)\bigr)\\
&=\Lambda_f(s,\tfrac{a}{q})
\bigl(\psi'(\tfrac{k+1}2-s)-\psi'(s+\tfrac{k-1}2)\bigr).
\end{align*}
Applying the classical Voronoi formula \cite[p.~179, (A.10)]{KMV02}
$$
\Lambda_f(s,\tfrac{a}{q})
=\epsilon\xi(q)(Nq^2)^{\frac12-s}
\Lambda_{\bar{f}}\bigl(1-s,-\tfrac{\overline{Na}}q\bigr),
$$
we arrive at \eqref{dstarfunceq}.

Finally, by \eqref{Deltafaq} and the nonvanishing of automorphic
$L$-functions for $\Re(s)\ge1$ \cite{JS76},
$\Delta_{f,a,q}^*(s)$ is holomorphic for $\Re(s)\ge1$. This conclusion
applies to $\Delta_{\bar{f},-\overline{Na},q}^*(s)$ as well,
so by \eqref{dstarfunceq}, all poles of $\Delta_{f,a,q}^*(s)$ have
real part in $(0,1)$.
\end{proof}

Fix, for the remainder of this section, a choice of $f,a,q$ as in
Proposition~\ref{voronoi}, and $\alpha\in\Q^\times$.
We define
$$
\Ns_{f,a,q}(T)=\#\bigl\{\rho\in\C:|\Im(\rho)|\le T,
\Res{s=\rho}\Delta_{f,a,q}^*(s)\ne0\bigr\}
$$
and
\begin{equation}\label{Sydefn}
S_{f,a,q}(y,\alpha)=
\sum_{\rho\in\C}\Res{s=\rho}\Delta_{f,a,q}^*(s)(y-i\alpha)^{-\rho-\frac{k-1}2}
\quad\text{for }y\in\R_{>0},
\end{equation}
where $(y-i\alpha)^{-\rho-\frac{k-1}2}$ is defined in terms of the principal
branch of $\log(y-i\alpha)$.
Our goal is to derive the following expression
for the Mellin transform of $S_{f,a,q}(y,\alpha)$, up to a holomorphic
function on $\{s\in\C:\Re(s)>0\}$:
\begin{proposition}\label{prop:Mellin}
Define
\begin{equation}\label{eq:Hdef}
H_{f,a,q,\alpha}(s)=
\Delta_{f,a,q}(s,\alpha)-\epsilon\xi(q)(i\sgn\alpha)^k
(Nq^2\alpha^2)^{s-\frac12}\Delta_{\bar{f},-\overline{Na},q}
\!\left(s,-\frac1{Nq^2\alpha}\right)
\end{equation}
and
$$
I_{f,a,q,\alpha}(s)=\int_0^{|\alpha|/4}
S_{f,a,q}(y,\alpha)y^{s+\frac{k-1}2}\frac{dy}{y}.
$$
Then $I_{f,a,q,\alpha}(s)-H_{f,a,q,\alpha}(s)$ has analytic continuation
to $\Re(s)>0$. Moreover, if
$$
\int_0^{|\alpha|/4}|S_{f,a,q}(y,\alpha)|
y^{\sigma+\frac{k-1}2}\frac{dy}{y}<\infty
$$
for some $\sigma\ge0$, then $H_{f,a,q,\alpha}(s)$ is holomorphic for
$\Re(s)>\sigma$.
\end{proposition}
The proof will be carried out in several lemmas, and involves the
following auxiliary functions defined on $\H=\{z\in\C:\Im(z)>0\}$:
$$
F(z)=2\sum_{n=1}^\infty c_{f,a,q}(n)n^{\frac{k-1}{2}}e(nz),
\quad\overline{F}(z)=2\sum_{n=1}^\infty
c_{\bar{f},-\overline{Na},q}(n)n^{\frac{k-1}{2}}e(nz),
$$
$$
A(z)=\frac1{2\pi i}\int_{\Re(s)=\frac{k}2}\Lambda_f(s,\tfrac{a}q)
\big(\psi'(s+\tfrac{k-1}{2})+\psi'(s-\tfrac{k-1}{2})\big)(-iz)^{-s-\frac{k-1}2}\,ds,
$$
and
$$
B(z)=\frac1{2\pi i}\int_{\Re(s)=\frac{k}2}\Lambda_f(s,\tfrac{a}{q})
\frac{\pi^2}{\sin^2(\pi(s+\tfrac{k-1}2))}(-iz)^{-s-\frac{k-1}2}\,ds.
$$

We first derive the following expression for $S_{f,a,q}$.
\begin{lemma}\label{Sfaqz}
For $z=\alpha+iy\in\H$, we have
\begin{equation}\label{eq:Sfaqz}
S_{f,a,q}(y,\alpha)=F(z)
-\frac{\epsilon\xi(q)}{(-i\sqrt{N}qz)^k}
\overline{F}\!\left(-\frac1{Nq^2z}\right)+A(z)-B(z).
\end{equation}
\end{lemma}
\begin{proof}
Let $0<\varepsilon<\frac{1}{2}$.
For $z\in\H$ we define
\[
I_R(z)=\frac1{2\pi i}\int_{\Re(s)=1+\varepsilon}\Delta_{f,a,q}(s)(-iz)^{-s-\frac{k-1}2}\,ds,
\quad
I_L(z)=\frac1{2\pi i}\int_{\Re(s)=-\varepsilon}\Delta_{f,a,q}(s)(-iz)^{-s-\frac{k-1}2}\,ds.
\]
For the remainder of the proof we let $z =\alpha+iy$.

Since $\Delta_{f,a,q}^*(s)$ is a
ratio of entire functions of finite order with at most simple poles,
by the calculus of residues we have
\[
\Res{s=0}\Delta_{f,a,q}(s)(-iz)^{-s-\frac{k-1}2}
+S_{f,a,q}(y,\alpha)=I_R(z)-I_L(z).
\]
Note that the residue term at $s=0$ vanishes unless $k=1$.
We have
\begin{align*}
I_R(z)&=\frac1{2\pi i}\int_{\Re(s)=1+\varepsilon}\Gamma_\C(s+\tfrac{k-1}{2})
D_{f,a,q}(s)(-iz)^{-s-\frac{k-1}2}\,ds\\
&=2(-2\pi iz)^{-\frac{k-1}2}\sum_{n=1}^{\infty} c_{f,a,q}(n)
\frac1{2\pi i}\int_{\Re(s)=1+\varepsilon}\Gamma(s+\tfrac{k-1}{2})
(-2\pi inz)^{-s}\,ds.
\end{align*}
Using the identity
$$
\frac1{2\pi i}\int_{\Re(s)=1+\varepsilon}\Gamma(s+\tfrac{k-1}{2})z^{-s}\,ds=
z^{\frac{k-1}{2}}e^{-z}\quad\text{for }\Re(z)>0,
$$
it follows that
\begin{equation}\label{IRidentity}
I_R(z)=2\sum_{n=1}^\infty c_{f,a,q}(n)n^{\frac{k-1}{2}}e(nz)=F(z).
\end{equation}
By the functional equation, we have
$$
\Delta_{f,a,q}(s)=\epsilon\xi(q)(Nq^2)^{\frac12-s}
\Delta_{\bar{f},-\overline{Na},q}(1-s)+\Lambda_f(s,\tfrac{a}{q})
\big( \psi'(\tfrac{k+1}{2}-s)-\psi'(s+ \tfrac{k-1}{2}) \big),
$$
so $I_L(z)=I_{L1}(z)+I_{L2}(z)$,
where
\begin{align}\label{IL1}
I_{L1}(z)&=\frac1{2\pi i}\int_{\Re(s)=-\varepsilon}
\epsilon\xi(q)(Nq^2)^{\frac{1}2-s}
\Delta_{\bar{f},-\overline{Na},q}(1-s)(-iz)^{-s-\frac{k-1}2}\,ds, \\
\label{IL2}
I_{L2}(z)&=\frac1{2\pi i}\int_{\Re(s)=-\varepsilon}
\Lambda_f(s,\tfrac{a}{q})
\big(\psi'(\tfrac{k+1}{2}-s)-\psi'(s+\tfrac{k-1}{2})\big)
(-iz)^{-s-\frac{k-1}2}\,ds.
\end{align}

Making the substitution $s\mapsto 1-s$ in \eqref{IL1}, we get
\begin{equation}\label{IL1identity}
\begin{aligned}
I_{L1}(z)&=\frac{1}{2\pi i} \int_{\Re(s)=1+\varepsilon}
\epsilon\xi(q)(Nq^2)^{s-\frac{1}2}
\Delta_{\bar{f},-\overline{Na},q}(s)(-iz)^{s-\frac{k+1}2}\,ds\\
&=2\epsilon\xi(q)(Nq^2)^{-\frac12}(-iz)^{-\frac{k+1}2}
(2\pi)^{-\frac{k-1}2}
\frac1{2\pi i}\int_{\Re(s)=1+\varepsilon}
\Delta_{\bar{f},-\overline{Na},q}(s)\Big(\frac{2\pi}{-iNq^2z}\Big)^{-s}\,ds\\
&=2\epsilon\xi(q)(Nq^2)^{-\frac12}(-iz)^{-\frac{k+1}2}(2\pi)^{-\frac{k-1}2}
\sum_{n=1}^\infty c_{\bar{f},-\overline{Na},q}(n)
\Big(\frac{2\pi n}{-iNq^2z}\Big)^{\frac{k-1}{2}}e\Big(-\frac{n}{Nq^2z}\Big)\\
&=\frac{2\epsilon\xi(q)}{(-i\sqrt{N}qz)^k}
\sum_{n=1}^\infty c_{\bar{f},-\overline{Na},q}(n)n^{\frac{k-1}{2}}
e\Big(-\frac{n}{Nq^2z}\Big)\\
&=\frac{\epsilon\xi(q)}{(-i\sqrt{N}qz)^k}
\overline{F}\!\left(-\frac1{Nq^2z}\right).
\end{aligned}
\end{equation}

Next, note that the integrand in \eqref{IL2} is holomorphic for
$-\frac{k-1}{2}<\Re(s)<\frac{k+1}{2}$.
Moving the contour to $\Re(s)=\frac{k}2$, we get a contribution from the
pole at $s=0$ (present only when $k=1$) of
$$
\Res{s=0}\Lambda_f(s,\tfrac{a}{q})\psi'(s+\tfrac{k-1}2)(-iz)^{-s-\frac{k-1}2}
=-\Res{s=0}\Delta_{f,a,q}(s,\tfrac{a}{q})(-iz)^{-s-\frac{k-1}2}.
$$
Thus
\begin{align*}
I_{L2}(z)+\Res{s=0}&\Delta_{f,a,q}(s,\tfrac{a}{q})(-iz)^{-s-\frac{k-1}2}\\
&=\frac1{2\pi i} \int_{\Re(s)=\frac{k}2}
\Lambda_f(s,\tfrac{a}{q})
\big(\psi'(\tfrac{k+1}2-s)-\psi'(s+\tfrac{k-1}2)\big)(-iz)^{-s-\frac{k-1}2}\,ds.
\end{align*}
The reflection formula for $\Gamma$ implies that
$\psi'(1-s)+\psi'(s)=\frac{\pi^2}{\sin^2(\pi s)}$, so
\[
\psi'(\tfrac{k+1}{2}-s)
-\psi'(s+\tfrac{k-1}{2})
=\frac{\pi^2}{\sin^2(\pi(s+\tfrac{k-1}{2}))}
-\psi'(s+ \tfrac{k-1}{2})-\psi'(s-\tfrac{k-1}{2}).
\]
Therefore
$I_{L2}(z)+\Res{s=0}\Delta_{f,a,q}(s,\tfrac{a}{q})(-iz)^{-s-\frac{k-1}2}
=I_{L2B}(z)-I_{L2A}(z)$, where
\begin{align*}
I_{L2A}(z)&=\frac1{2\pi i}\int_{\Re(s)=\frac{k}2}\Lambda_f(s,\tfrac{a}{q})
\big(\psi'(s+\tfrac{k-1}{2})+\psi'(s-\tfrac{k-1}{2})\big)(-iz)^{-s-\frac{k-1}2}\,ds,\\
I_{L2B}(z)&=\frac1{2\pi i}\int_{\Re(s)=\frac{k}2}
\Lambda_f(s,\tfrac{a}{q})\frac{\pi^2}{\sin^2(\pi(s+\tfrac{k-1}2))}(-iz)^{-s-\frac{k-1}2}\,ds.
\end{align*}
Hence
$$
S_{f,a,q}(y,\alpha)=I_R(z)-I_{L1}(z)+I_{L2A}(z)-I_{L2B}(z).
$$
By applying \eqref{IRidentity}, \eqref{IL1identity} and by setting
$A(z)=I_{L2A}(z)$ and $B(z)= I_{L2B}(z)$,
we establish Lemma \ref{Sfaqz}.
\end{proof}

Next we evaluate
$\int_0^{|\alpha|/4}S_{f,a,q}(y,\alpha)y^{s+\frac{k-1}2}\frac{dy}{y}$,
considering each term on the right-hand side of \eqref{eq:Sfaqz} in
turn.
\begin{lemma}\label{Flemma}
$\int_0^{|\alpha|/4}F(\alpha+iy)y^{s+\frac{k-1}2}\frac{dy}{y}
-\Delta_{f,a,q}(s,\alpha)$
continues to an entire function of $s$.
\end{lemma}
\begin{proof}
From the definition of $F$ we compute that
$$
\int_0^\infty F(\alpha+iy)y^{s+\frac{k-1}2}\frac{dy}{y}
=\Delta_{f,a,q}(s,\alpha).
$$
Moreover, $F(\alpha+iy)$ decays exponentially as $y\to\infty$, so
the contribution to the integral from $y>|\alpha|/4$ is entire.
\end{proof}

\begin{lemma}\label{Fbarlemma}
For any $M\in\Z_{\ge0}$,
\begin{equation}\label{eq:fbarmellin}
\begin{aligned}
&\int_0^{|\alpha|/4}\bigl(-i\sqrt{N}q(\alpha+iy)\bigr)^{-k}
\overline{F}\!\left(-\frac1{Nq^2(\alpha+iy)}\right)
y^{s+\frac{k-1}2}\frac{dy}{y}\\
&-(i\sgn\alpha)^k\sum_{m=0}^{M-1}(-i\alpha)^{-m}
{{s+m-\frac{k+1}2}\choose{m}}(Nq^2\alpha^2)^{s-\frac12+m}
\Delta_{\bar{f},-\overline{Na},q}\!\left(s+m,-\frac1{Nq^2\alpha}\right)
\end{aligned}
\end{equation}
continues to a holomorphic function on $\{s\in\C:\Re(s)>1-M\}$.
\end{lemma}
\begin{proof}
As the proof of this lemma is very similar to that of \cite[Lemma~3.3]{Boo16},
we just provide a sketch and refer
to the appropriate parts of loc.~cit.\ for the relevant details.
Fix $y\in(0,|\alpha|/4]$, and set $z=\alpha+iy$, $\beta=-1/Nq^2\alpha$,
and $u=y/\alpha$. It may be checked that
\[
-\frac1{Nq^2z}=\beta+i|\beta u|-\frac{\beta u^2}{1+iu}.
\]
Therefore
\begin{align*}
&(-i\sqrt{N}qz)^{-k}\overline{F}\!\left(-\frac1{N q^2 z}\right)\\
&=2(-i\sqrt{N}q\alpha)^{-k}\sum_{n=1}^\infty
c_{\bar{f},-\overline{Na},q}(n)n^{\frac{k-1}2}e(\beta n)e^{-2\pi n|\beta u|}
(1+iu)^{-k}e\Big(-\frac{n\beta u^2}{1+iu}\Big).
\end{align*}
It was shown in \cite[p.~820]{Boo16} that
\[
(1+iu)^{-k}e\Big(-\frac{n\beta u^2}{1+iu}\Big)
=\sum_{m=0}^\infty(-iu)^m\sum_{j=0}^m\binom{m+k-1}{m-j}
\frac{(-2\pi n|\beta u|)^j}{j!},
\]
and for $M,K\in\Z_{\ge 0}$ and $|u|\le\frac14$,
\[
\sum_{m=M}^\infty(-iu)^m\sum_{j=0}^m\binom{m+k-1}{m-j}
\frac{(-2\pi n|\beta u|)^j}{j!}
\ll_{\alpha,M,K}|u|^{M-K}n^{-K}e^{2\pi n|\beta u|}.
\]
Thus, we obtain
\begin{equation}\label{Fbaridentity}
\begin{aligned}
&(-i\sqrt{N}qz)^{-k}\overline{F}\!\left(-\frac1{Nq^2z}\right)
=O_{M,K}\Big(y^{M-K}\sum_{n=1}^\infty|c_{\bar{f},-\overline{Na},q}(n)|n^{\frac{k-1}{2}-K}\Big)
+2(-i\sqrt{N}q\alpha)^{-k}\\
&\times\sum_{m=0}^{M-1}\Big(-\frac{iy}{\alpha}\Big)^m
\sum_{j=0}^m\binom{m+k-1}{m-j}
\sum_{n=1}^\infty c_{\bar{f},-\overline{Na},q}(n)n^{\frac{k-1}2}e(\beta n)
\frac1{j!}\Big(-\frac{2\pi ny}{Nq^2\alpha^2}\Big)^j
e^{-\frac{2\pi ny}{Nq^2\alpha^2}}.
\end{aligned}
\end{equation}
By the choice $K=\lfloor\frac{k-1}2\rfloor+2$, the error term converges
and is $O_M(y^{M-K})$.
For the other term note that
\begin{align*}
&2y^m\sum_{n=1}^\infty c_{\bar{f},-\overline{Na},q}(n)n^{\frac{k-1}2}e(\beta n)
\frac1{j!}\Big(-\frac{2\pi ny}{Nq^2\alpha^2}\Big)^j
e^{-\frac{2\pi ny}{Nq^2\alpha^2}}\\
&=2\frac{y^{j+m}}{j!}\frac{d^j}{dy^j}
\sum_{n=1}^\infty c_{\bar{f},-\overline{Na},q}(n)n^{\frac{k-1}2}e(\beta n)
e^{-\frac{2\pi ny}{Nq^2\alpha^2}}\\
&=\frac{y^{j+m}}{j!}\frac{d^j}{dy^j}
\frac1{2\pi i}\int_{\Re(s)=m+2}
(Nq^2\alpha^2)^{s+\frac{k-1}2}\Delta_{\bar{f},-\overline{Na},q}(s,\beta)y^{-s-\frac{k-1}2}\,ds\\
&=\frac1{2\pi i}\int_{\Re(s)=2}
\binom{-s-\frac{k-1}2-m}{j}
(Nq^2\alpha^2)^{s+\frac{k-1}2+m}
\Delta_{\bar{f},-\overline{Na},q}(s+m,\beta)y^{-s-\frac{k-1}2}\,ds.
\end{align*}
Inserting this in the last term of \eqref{Fbaridentity} and using the
Chu--Vandermonde identity
\[
\sum_{j=0}^m\binom{m+k-1}{m-j}\binom{-s-\frac{k-1}2-m}{j}=
\binom{-s+\frac{k-1}2}{m}=(-1)^m\binom{s+m-\frac{k+1}2}{m},
\]
we arrive at
\begin{align*}
&(-i\sqrt{N}qz)^{-k}\overline{F}\!\left(-\frac1{Nq^2z}\right)\\
&=O_M(y^{M-\lfloor\frac{k+3}2\rfloor})
+(i\sgn\alpha)^k\sum_{m=0}^{M-1}\frac{(-i\alpha)^{-m}}{2\pi i}\int_{\Re(s)=2}
{{s+m-\frac{k+1}2}\choose{m}}
(Nq^2\alpha^2)^{s-\frac12+m}\\
&\hspace{8cm}\cdot
\Delta_{\bar{f},-\overline{Na},q}\!\left(s+m,-\frac1{Nq^2\alpha}\right)
y^{-s-\frac{k-1}2}\,ds.
\end{align*}

We multiply both sides by $y^{s+\frac{k-1}2-1}$ and integrate over
$y\in(0,|\alpha|/4]$. The error term yields a holomorphic function
for $\Re(s)>2-M$. As for the sum over $m$, by shifting the contour
to the right, we see that each term decays rapidly as $y\to\infty$,
so the integral over $(0,|\alpha|/4]$ differs from the full Mellin
transform by an entire function. By Mellin inversion, it follows
that \eqref{eq:fbarmellin} is holomorphic for $\Re(s)>2-M$. Finally,
replacing $M$ by $M+1$ and discarding the final term of the sum concludes
the proof of the lemma.
\end{proof}

\begin{lemma}\label{Alemma}
$\Gamma_\C(s)^{-1}\int_0^{|\alpha|/4}A(\alpha+iy)y^s\frac{dy}y$
continues to an entire function of $s$.
\end{lemma}
\begin{proof}
Let $\Phi(s)=\psi'(s+\tfrac{k-1}2)+\psi'(s-\tfrac{k-1}2)$.
By the identity $\psi'(s)=\int_1^\infty\frac{\log x}{x-1}x^{-s}\,dx$,
we have $\Phi(s)=\int_1^\infty\phi(x)x^{-s-\frac{k-1}2}\,dx$
for $\Re(s)>\frac{k-1}{2}$, where $\phi(x)=\frac{x^{k-1}+1}{x-1}\log{x}$.
It follows that
$$
\Phi(s)\Gamma(s+\tfrac{k-1}{2})=\int_1^\infty\phi(x)x^{-s-\frac{k-1}2}\,dx
\int_0^\infty e^{-y}y^{s+\frac{k-1}2}\,dy
=\int_1^\infty\phi(x)\int_0^\infty e^{-y}
\Big(\frac{y}{x}\Big)^{s+\frac{k-1}{2}}\frac{dy}{y}.
$$
By the variable change $y\mapsto xy$ we obtain
$$
\Phi(s)\Gamma(s+\tfrac{k-1}{2})=
\int_1^\infty\phi(x)\int_0^\infty e^{-xy}y^{s+\frac{k-1}{2}}\frac{dy}{y}
=\int_0^\infty\Big(\int_1^\infty\phi(x)e^{-xy}\,dx\Big)y^{s+\frac{k-1}{2}}\frac{dy}{y}.
$$
By Mellin inversion,
\begin{equation}\label{mellininversion}
\int_1^\infty\phi(x)e^{-xy}\,dx
=\frac1{2\pi i}\int_{\Re(s)=2}\Phi(s)\Gamma(s+\tfrac{k-1}2)y^{-s-\frac{k-1}2}\,ds.
\end{equation}
Observe that $L_{\bar{f}}(s,-\frac{\overline{Na}}{q})
=\sum_{n=1}^\infty b_nn^{-s}$,
where $b_n=\lambda_{\bar{f}}(n)e(-\frac{\overline{Na}}{q})$.
Thus for $z\in\mathbb{H}$,
\begin{align*}
A(z)&=2\sum_{n=1}^\infty b_n\cdot
\frac1{2\pi i}\int_{\Re(s)=2}\Phi(s)
\Gamma(s+\tfrac{k-1}{2})(-2\pi inz)^{-s-\frac{k-1}2}\,ds\\
&=2\sum_{n=1}^\infty b_n\int_1^\infty\phi(x)e^{2\pi inxz}\,dx,
\end{align*}
where the last step follows from \eqref{mellininversion}. For $z=\alpha+iy$
this simplifies to
\begin{equation}\label{eq:Aalphaiy}
A(\alpha+iy)=2\sum_{n=1}^\infty b_n\int_1^\infty\phi(x)
e(\alpha nx)e^{-2\pi nxy}\,dx.
\end{equation}
Using this expression, it follows that
\begin{equation}\label{Amellintransform}
\begin{aligned}
&\int_0^\infty A(\alpha+iy)y^s\frac{dy}{y}
=2\sum_{n=1}^\infty b_n\int_1^\infty\phi(x)e(\alpha nx)
\int_0^\infty e^{-2\pi nxy}y^s\frac{dy}{y}\,dx\\
&=\Gamma_\C(s)\sum_{n=1}^\infty b_nn^{-s}
\int_1^\infty\phi(x)e(\alpha nx)x^{-s}\,dx.
\end{aligned}
\end{equation}

For $j=0,1,2,\ldots$, define the sequence of functions $\phi_j(x,s)$ by
\[
\phi_0(x,s)=\phi(x), \quad
\phi_{j+1}(x,s)=x\frac{\partial\phi_j}{\partial x}(x,s)-(s+j)\phi_j(x,s).
\]
Integrating by parts,
\[
\int_1^\infty\phi_j(x,s)e(\alpha nx)x^{-s-j}\,dx
=-\frac{e(\alpha n)\phi_j(1,s)}{2\pi i\alpha n}
-\frac1{2\pi i\alpha n}\int_1^\infty\phi_{j+1}(x,s)e(\alpha nx)
x^{-s-j-1}\,dx.
\]
Repeated application of this yields
\begin{equation}\label{intparts}
\begin{aligned}
\int_1^\infty\phi(x)e(\alpha nx)x^{-s}\,dx
&=e(\alpha n)\sum_{j=0}^{m-1}\frac{\phi_j(1,s)}{(-2\pi i\alpha n)^{j+1}}\\
&+(-2\pi i\alpha n)^{-m}\int_1^\infty\phi_m(x,s)
e(\alpha nx)x^{-s-m}\,dx
\end{aligned}
\end{equation}
for $m\in\Z_{\ge0}$. By \eqref{Amellintransform} and \eqref{intparts}
it follows that
\begin{align*}
\frac1{\Gamma_\C(s)}
\int_0^\infty A(\alpha+iy)y^s\frac{dy}{y}
&=\sum_{j=0}^{m-1}\frac{\phi_j(1,s)}{(-2\pi i\alpha)^{j+1}}
L_{\bar{f}}(s+j+1,-\tfrac{\overline{Na}}{q}+\alpha)\\
&+(-2\pi i\alpha)^{-m}\sum_{n=1}^\infty\frac{b_n}{n^{s+m}}
\int_1^\infty\phi_m(x,s)e(\alpha nx)x^{-s-m}\,dx.
\end{align*}
Each term in the sum extends to an entire function of $s$, by
\cite[Proposition~3.1]{BK11}. Furthermore,
it may be checked that $\phi_m(x,s)\ll_{m,k}(1+|s|)^mx^{k-1}$. Therefore
the last integral is holomorphic for $\Re(s)>k-m$. Letting
$m\to\infty$ shows that
$\Gamma_\C(s)^{-1}\int_0^\infty A(\alpha+iy)y^s\frac{dy}y$
continues to an entire function.

Finally, from \eqref{eq:Aalphaiy} we see that $A(\alpha+iy)$ decays
exponentially as $y\to\infty$, and hence
$\int_{|\alpha|/4}^\infty A(\alpha+iy)y^s\frac{dy}{y}$ is entire.
This completes the proof.
\end{proof}

\begin{lemma}\label{Blemma}
$\Gamma_\C(s)^{-1}\int_0^{|\alpha|/4}B(\alpha+iy)y^s\frac{dy}y$
continues to an entire function of $s$.
\end{lemma}
\begin{proof}
Following the proof of \cite[Lemma~3.4]{Boo16}, we obtain
$$
B(\alpha+iy)=\sum_{j=0}^{M-1}P_j(\alpha)y^j+O_M(y^M)
\quad\text{for all }M\in\Z_{\ge0}, y\in\bigl(0,\tfrac{|\alpha|}4\bigr],
$$
where
$$
P_j(\alpha)=\frac{(-i\alpha)^{-j}}{2\pi i}\int_{\Re(s)=\frac{k}2}
e^{i\frac{\pi}2\sgn(\alpha)(s+\frac{k-1}2)}|\alpha|^{-s-\frac{k-1}2}
\binom{-s-\frac{k-1}2}{j}\Lambda_f(s,\tfrac{a}{q})
\frac{\pi^2}{\sin^2(\pi(s+\frac{k-1}2))}\,ds.
$$
Hence,
$$
\int_0^{|\alpha|/4}B(\alpha+iy)y^s\frac{dy}{y}
-\sum_{j=0}^{M-1}P_j(\alpha)\frac{|\alpha/4|^{s+j}}{s+j}
$$
is holomorphic for $\Re(s)>-M$. Note that the sum over $j$ is entire
apart from at most simple poles at the poles of $\Gamma_\C(s)$.
Dividing by $\Gamma_\C(s)$ and taking $M\to\infty$ concludes the proof.
\end{proof}

\begin{proof}[Proof of Proposition~\ref{prop:Mellin}]
Combining Lemmas~\ref{Sfaqz}--\ref{Blemma} and taking $M=1$,
we see that $I_{f,a,q,\alpha}(s)-H_{f,a,q,\alpha}(s)$ has analytic
continuation to $\Re(s)>0$. If
$\int_0^{|\alpha|/4}|S_{f,a,q}(y,\alpha)|
y^{\sigma+\frac{k-1}2}\frac{dy}{y}<\infty$ for some
$\sigma\ge0$, then the integral defining $I_{f,a,q,\alpha}(s)$ converges
absolutely for $\Re(s)>\sigma$, and hence $I_{f,a,q,\alpha}(s)$ is
holomorphic in that region.
\end{proof}

\section{Estimates for $\Ns_{f,a,q}(T)$}
Fix $f,a,q$ as in Proposition~\ref{voronoi}, and let
$\alpha\in\Q^\times$. In this section, we derive estimates for
$\Ns_{f,a,q}(T)$ based on Proposition~\ref{prop:Mellin}.
\begin{lemma}\label{GLprimebound}
Let $f\in S_k(\Gamma_1(N))$ be a primitive form.
For $\rho=\beta+i\gamma$ a zero of $\Lambda_f(s)$, we have
\begin{equation}\label{compLfncboundrho}
\Lambda_f'(\rho)\ll_f
(2+|\gamma|)^{\frac{k}{2}+\frac{|\beta-\frac12|}{3}-\frac16}\log^2(2+|\gamma|)
e^{-\frac{\pi}{2}|\gamma|}.
\end{equation}
\end{lemma}
\begin{proof}
We begin by establishing, for $s=\sigma+it$ and $\sigma\in[\frac12,1]$,
\begin{equation}\label{compLfncbound}
\Gamma_\C(s+\tfrac{k-1}{2})L_f'(s) \ll
\tau^{\frac{k}{2}-\frac{1-\sigma}{3}}
e^{-\frac{\pi}{2}|t|}\log^2\tau,
\end{equation}
where $\tau=|t|+2$. By \cite[Theorem 1.1]{BMN19}, we have
$$
L_f(\tfrac12+it)\ll\tau^{\frac13}\log\tau.
$$
By the Phragm\'{e}n--Lindel\"{o}f principle, using
$$
L_f\!\left(-\frac1{\log\tau}+it\right)\ll\tau\log\tau
\quad\text{and}\quad
L_f\!\left(1+\frac1{\log \tau}+it\right)\ll\log\tau,
$$
it follows that
$L_f(\sigma+it)\ll\tau^{\frac13}\log\tau$
when $|\sigma-\tfrac12|\le1/\log\tau$.
An application of the Cauchy integral formula then
yields $L_f'(\tfrac12+it)\ll\tau^{\frac13}\log^2\tau$.
By Cauchy's inequality and Rankin's estimate
$\sum_{n\le x}|\lambda_f(n)|^2\ll x$, we get
$$
|L_f'(1+\varepsilon+it)|
\le\sum_{n=1}^\infty\frac{|\lambda_f(n)|\log n}{n^{1+\varepsilon}}
\le\Big(\sum_{n=1}^\infty\frac{|\lambda_f(n)|^2}{n^{1+\varepsilon}}\Big)^{\frac12}
\zeta''(1+\varepsilon)^{\frac12}
\ll\varepsilon^{-\frac12}\varepsilon^{-\frac32}=\varepsilon^{-2}
$$
for $\varepsilon>0$.
Another application of the Phragm\'{e}n--Lindel\"{o}f principle yields
$$
L_f'(\sigma+it)\ll\tau^{\frac23(1-\sigma)}\log^2\tau
$$
for $\sigma\in[\frac12,1]$. This, together with the Stirling formula estimate
$$
\Gamma_\C(s+\tfrac{k-1}{2})\ll \tau^{\sigma+\tfrac{k}{2}-1}
e^{-\frac{\pi}{2}|t|},
$$
yields \eqref{compLfncbound}.

Setting $s=\rho=\beta+i\gamma$ with
$\beta\ge\frac12$ in \eqref{compLfncbound} gives
\begin{equation}\label{lambdafprimerhochi}
\Lambda_f'(\rho)
=\Gamma_\C(\rho+\tfrac{k-1}{2})L_f'(\rho)
\ll(2+|\gamma|)^{\frac{k}{2}-\frac{1-\beta}{3}}
\log^2(2+|\gamma|)e^{-\frac{\pi}{2}|\gamma|}.
\end{equation}
Now suppose $\beta<\frac12$.
Differentiating the functional equation we obtain
$$
\Lambda_f'(\rho)=-\epsilon N^{\frac12-\rho}\Lambda_{\bar{f}}'(1-\rho),
$$
where $|\epsilon|=1$.  Applying \eqref{lambdafprimerhochi} to
$\Lambda_{\bar{f}}'(1-\rho)$ it follows that
\begin{equation}\label{lambdafprimerhochi2}
\Lambda_f'(\rho)
\ll(2+|\gamma|)^{\frac{k}{2}-\frac{\beta}{3}}
\log^2(2+|\gamma|)e^{-\frac{\pi}{2}|\gamma|}.
\end{equation}
Combining \eqref{lambdafprimerhochi} and \eqref{lambdafprimerhochi2} we obtain
\eqref{compLfncboundrho}.
\end{proof}

\begin{lemma}\label{lem:Nathan}
For any fixed $\varepsilon>0$ and all $\sigma\in[\varepsilon,2]$,
$$
\int_0^{\frac{|\alpha|}{4}}|S_{f,a,q}(y,\alpha)|y^{\sigma+\frac{k-1}{2}}\frac{dy}{y}\ll
\sum_{\substack{\rho=\beta+i\gamma\\\text{a pole of }\Delta_{f,a,q}^*(s)}}
(2+|\gamma|)^{\frac{1+|\beta-\frac12|}3-\sigma}\log^2(2+|\gamma|).
$$
\end{lemma}
\begin{proof}
Throughout this proof we let $\rho=\beta+i\gamma$ denote a pole of
$\Delta_{f,a,q}^*(s)$, and we set $\tau=2+|\gamma|$.
Recalling \eqref{Sydefn}, observe that
$(y-i\alpha)^{-\rho-\frac{k-1}2}=e^{i\frac{\pi}{2}\sgn(\alpha)(\rho+\frac{k-1}2)}
|\alpha|^{-\rho-\frac{k-1}2}(1+\frac{iy}{\alpha})^{-\rho-\frac{k-1}2}$
and
\begin{equation}\label{1plusyalpha}
\begin{aligned}
\bigl|(1+i\tfrac{y}{\alpha})^{-(\beta+i\gamma+\frac{k-1}2)}\bigr|
&=\bigl|e^{-(\frac12\log(1+(y/\alpha)^2)+i\arctan(y/\alpha))(\beta+i\gamma+\frac{k-1}2)}\bigr|\\
&=\bigl(1+(\tfrac{y}{\alpha})^2\bigr)^{-\frac{\beta}{2}-\frac{k-1}4}
e^{\gamma\arctan(y/\alpha)}.
\end{aligned}
\end{equation}
Therefore
\begin{equation}\label{yalphabound}
(y-i\alpha)^{-\rho-\frac{k-1}2}
\ll e^{\gamma\sgn(\alpha)(\arctan(y/|\alpha|)-\frac{\pi}2)}.
\end{equation}
Next, we treat the residue in \eqref{Sydefn}.
By \eqref{Deltafaq}, the poles of $\Delta_{f,a,q}^*(s)$ arise from
poles of $\Delta_f(s)$ and $\Delta_f(s,\chi)$ with $\chi\ne\chi_0$.
The contributrion of an individual term of \eqref{Deltafaq} to
$\Res{s=\rho}\Delta_{f,a,q}^*(s)$, if nonzero, is of the form
$$
-\Big(1-\frac{q}{q-1}P_{f,q}(q^{-\rho})\Big)\Lambda_f'(\rho)
\quad\text{or}\quad
-\frac{\tau(\overline{\chi})\chi(a)}{q-1}\Lambda_f'(\rho,\chi).
$$
Applying Lemma~\ref{GLprimebound} (possibly replacing $f$ by
$f\otimes\chi$) to each of these expressions, it follows that
\begin{equation}\label{residue}
\Res{s=\rho}\Delta_{f,a,q}^*(s)
\ll\tau^{\frac{k}2+\frac{|\beta-\frac12|}3-\frac16}(\log^2\tau)
e^{-\frac{\pi}{2}|\gamma|}.
\end{equation}
It follows from \eqref{Sydefn}, \eqref{yalphabound}, and \eqref{residue}  that
$$
S_{f,a,q}(y,\alpha)\ll\sum_\rho
\tau^{\frac{k}2+\frac{|\beta-\frac12|}3-\frac16}(\log^2\tau)
e^{|\gamma|[\sgn(\alpha\gamma)\arctan(y/|\alpha|)-\frac{\pi}2(1+\sgn(\alpha\gamma))]}.
$$
By considering cases and using the bound $\arctan{u}\ge\frac{u}{2}$
for $0\le u\le\frac14$, we have
$$
S_{f,a,q}(y,\alpha)\ll\sum_\rho
\tau^{\frac{k}2+\frac{|\beta-\frac12|}3-\frac16}(\log^2\tau)e^{-c|\gamma|y}
\quad\text{for }y\in\bigl(0,\tfrac{|\alpha|}4\bigr],
$$
where $c=\frac1{2|\alpha|}>0$. We deduce from this
$$
\int_0^{\frac{|\alpha|}{4}}|S_{f,a,q}(y,\alpha)|
y^{\sigma+\frac{k-1}{2}}\frac{dy}{y}
\ll\sum_\rho
\tau^{\frac{k}2+\frac{|\beta-\frac12|}3-\frac16}\log^2\tau
\int_0^{\frac{|\alpha|}{4}}e^{-c|\gamma|y}y^{\sigma+\frac{k-1}{2}}\frac{dy}{y}.
$$
Now
$$
\int_0^{\frac{|\alpha|}{4}}e^{-c|\gamma|y}y^{\sigma+\frac{k-1}{2}}\frac{dy}{y}
\ll\int_0^{\frac{|\alpha|}{4}}e^{-c\tau y}y^{\sigma+\frac{k-1}{2}}\frac{dy}{y}
\le\int_0^\infty e^{-c\tau y}y^{\sigma+\frac{k-1}{2}}\frac{dy}{y}.
$$
By the variable change $u=c\tau y$, the last integral equals
$$
\frac1{(c\tau)^{\sigma+\frac{k-1}{2}}}\int_0^\infty e^{-u}u^{\sigma+\frac{k-1}{2}}\frac{du}{u}
=\frac{\Gamma(\sigma+\frac{k-1}{2})}{(c\tau)^{\sigma+\frac{k-1}{2}}}
\ll\tau^{-\sigma-\frac{k-1}2},
$$
and thus
$$
\int_0^{\frac{|\alpha|}{4}}S_{f,a,q}(y,\alpha)y^{\sigma+\frac{k-1}{2}}\frac{dy}{y}
\ll\sum_\rho
\tau^{\frac{1+|\beta-\frac12|}3-\sigma}\log^2\tau.
$$
\end{proof}

For a meromorphic function $h$ on $\{s\in\C:\Re(s)>1\}$, define
$$
\Theta(h)=\inf\bigl\{\theta\ge0:h\text{ continues analytically to }
\{s\in\C:\Re(s)>\theta\}\bigr\}.
$$
We also set
$$
\theta_{f,a,q}(T)=\sup\bigl(\{0\}\cup\bigl\{\Re(\rho),1-\Re(\rho):
\rho\in\C, |\Im(\rho)|\le T, \Res{s=\rho}\Delta_{f,a,q}^*(s)\ne0\bigr\}\bigr)
$$
and
$$
\theta_{f,a,q}=\lim_{T\to\infty}\theta_{f,a,q}(T).
$$
By Proposition~\ref{voronoi}, we have
\begin{equation}\label{thetafaq}
\theta_{f,a,q}=\max(\Theta(\Delta_{f,a,q}),\Theta(\Delta_{\bar{f},-\overline{Na},q})).
\end{equation}

\begin{proposition}\label{NfaqTlowerbound}
If $\Theta(H_{f,a,q,\alpha})>0$ then $\theta_{f,a,q}\ge\frac12$ and
\begin{equation}\label{eq:omega1}
\Ns_{f,a,q}(T)=
\Omega\bigl(T^{\frac13(1-\theta_{f,a,q})+\Theta(H_{f,a,q,\alpha})
-\frac12-\varepsilon}\bigr)
\quad\text{for all }\varepsilon>0.
\end{equation}
Further, if $\Theta(H_{f,a,q,\alpha})=\frac12$ and
$H_{f,a,q,\alpha}(s)$ has a pole with real part $\frac12$, then
\begin{equation}\label{eq:omega2}
\Ns_{f,a,q}(T)=\Omega\!\left(
\frac{T^{\frac13(1-\theta_{f,a,q}(T))}}{(1-\theta_{f,a,q}(T))\log^2{T}}\right),
\end{equation}
and there are arbitrarily large $T>0$ such that
\begin{equation}\label{eq:omega3}
\Ns_{f,a,q}(T)\ge\log\log\log{T}.
\end{equation}
\end{proposition}
\begin{proof}
Let $\beta_n+i\gamma_n$ run through the poles of $\Delta_{f,a,q}^*(s)$,
in increasing order of $|\gamma_n|$. For brevity, we write $I(s)$,
$H(s)$, $\Theta$, $S(y)$, $N(t)$, $\theta(t)$ and $\theta$ for
$I_{f,a,q,\alpha}(s)$,
$H_{f,a,q,\alpha}(s)$, $\Theta(H_{f,a,q,\alpha})$, $S_{f,a,q}(y,\alpha)$,
$\Ns_{f,a,q}(t)$, $\theta_{f,a,q}(t)$ and $\theta_{f,a,q}$, respectively.
By Lemma~\ref{lem:Nathan}, we have
\begin{align*}
\int_0^{|\alpha|/4}|S(y)|y^{\sigma+\frac{k-1}2}\frac{dy}{y}
&\ll\sum_{n\ge 1}(2+|\gamma_n|)^{\frac{1+|\beta_n-\frac12|}{3}-\sigma}
\log^2(2+|\gamma_n|)\\
&\le\sum_{n\ge 1}(2+|\gamma_n|)^{\frac{\theta(|\gamma_n|)}{3}
+\frac16-\sigma}\log^2(2+|\gamma_n|).
\end{align*}
If $\Theta>0$ then by Proposition~\ref{prop:Mellin}, 
the integral must diverge for sufficiently small $\sigma>0$, and thus
the right-hand side has infinitely many terms. Thus
$\Delta_{f,a,q}^*(s)$ has poles, so $\theta\ge\frac12$.

Suppose that \eqref{eq:omega1} does not hold. Then there exists
$\varepsilon\in(0,\Theta)$ such that
$N(t)=o(t^{\frac13(1-\theta)+\Theta-\frac12-\varepsilon})$.
Choosing $\sigma=\Theta-\frac{\varepsilon}3$ and
using the estimate
$\log^2(2+|\gamma_n|)\ll(2+|\gamma_n|)^{\frac{\varepsilon}{3}}$,
we have
\begin{align*}
\int_0^{|\alpha|/4}|S(y)|y^{\Theta-\frac{\varepsilon}3+\frac{k-1}2}\frac{dy}{y}
&\ll\sum_{n\ge1}
(2+|\gamma_n|)^{\frac{\theta}{3}+\frac16-\Theta+\frac23\varepsilon}
\ll1+\int_1^\infty
t^{\frac{\theta}{3}+\frac16-\Theta+\frac23\varepsilon}\,dN(t)\\
&\ll1+\int_1^\infty
t^{\frac{\theta}{3}+\frac16-\Theta+\frac23\varepsilon-1}N(t)\,dt
\ll1+\int_1^\infty t^{-1-\frac{\varepsilon}3}\,dt
\ll1.
\end{align*}
By Proposition~\ref{prop:Mellin}, it follows that
$H(s)$ is holomorphic for $\Re(s)>\Theta-\frac{\varepsilon}3$.
This is a contradiction, so \eqref{eq:omega1} must hold.

Next suppose that $\Theta=\frac12$, and
let $\rho$ be a pole of $H(s)$ with
$\Re(\rho)=\frac12$. Then for sufficiently small $\delta>0$,
by Proposition~\ref{prop:Mellin}, we have
$$
\delta^{-1}\ll|H(\rho+\delta)|\ll1+|I(\rho+\delta)|
\le1+\int_0^{|\alpha|/4}|S(y)|y^{\delta+\frac{k}2}\frac{dy}{y},
$$
where we understand the right-hand side to be $\infty$ if the integral
diverges. Applying Lemma~\ref{lem:Nathan}, we thus have
\begin{equation}\label{eq:gammasum}
\delta^{-1}
\ll1+\sum_{n\ge 1}(2+|\gamma_n|)^{\frac{\theta(|\gamma_n|)-1}{3}-\delta}
\log^2(2+|\gamma_n|).
\end{equation}
In particular, the right-hand side must have infinitely many terms.
Applying integration by parts, we get
\begin{align*}
\delta^{-1}&\ll1+\int_1^\infty
t^{\frac{\theta(t)-1}{3}-\delta}\log^2{t}\,dN(t)
=1-\int_1^\infty N(t)d(t^{\frac{\theta(t)-1}{3}-\delta}\log^2{t})\\
&\le1+\int_1^\infty N(t)\bigl(\tfrac{1-\theta(t)}{3}+\delta\bigr)
t^{\frac{\theta(t)-1}{3}-\delta-1}\log^2{t}\,dt,
\end{align*}
where for the last inequality we have used the fact that $\theta(t)$ is
nondecreasing and
$$
d(t^{\frac{\theta(t)-1}{3}-\delta}\log^2{t})
=t^{\frac{\theta(t)-1}{3}-\delta-1}(\log{t})
\bigl[2-\bigl(\tfrac{1-\theta(t)}{3}+\delta\bigr)\log{t}\bigr]\,dt
+\tfrac13\log^3{t}\,d\theta(t).
$$

Suppose that \eqref{eq:omega2} is false, so that the function
$\varepsilon(t)=N(t)t^{\frac13(\theta(t)-1)}(1-\theta(t))\log^2{t}$
satisfies $\lim_{t\to\infty}\varepsilon(t)=0$. Then we have
$$
\delta^{-1}\ll1+\int_1^\infty
\left(\frac13+\frac{\delta}{1-\theta(t)}\right)
\varepsilon(t)t^{-1-\delta}\,dt.
$$
By the standard zero-free region \cite[Theorem~5.10]{IK04}, we have
$$
\frac1{1-\theta(t)}\ll\log\max(t,2),
$$
so that
$$
\delta^{-1}\ll1+\int_1^\infty(1+\delta\log{t})\varepsilon(t)t^{-1-\delta}\,dt
=1+\delta^{-1}\int_0^\infty\varepsilon(e^{u/\delta})(1+u)e^{-u}\,du
=o(\delta^{-1}).
$$
This is a contradiction, so \eqref{eq:omega2} holds.

Finally, suppose \eqref{eq:omega3} is false, so that $N(T)<\log\log\log{T}$
for all sufficiently large $T$. Then there exists $n_0\ge\Z_{>0}$ such that
$|\gamma_n|>\exp\exp\exp{n}$ for all $n\ge n_0$.
Since the terms from $n<n_0$ contribute a bounded amount to
\eqref{eq:gammasum}, we have
$$
1+\sum_{n=n_0}^\infty|\gamma_n|^{-\delta}\log^2|\gamma_n|
\gg\delta^{-1}
$$
for all sufficiently small $\delta>0$.

Next we claim that there are infinitely many $m\ge n_0$ such that
\begin{equation}\label{eq:biggap}
\log\log|\gamma_{m+1}|\ge\tfrac{13}{5}\log\log|\gamma_m|.
\end{equation}
If not then there exists $n_1\ge n_0$ such that \eqref{eq:biggap} fails
for all $m\ge n_1$, and by induction it follows that
$$
\log\log|\gamma_n|\le(\tfrac{13}{5})^{n-n_1}\log\log|\gamma_{n_1}|
=c(\tfrac{13}{5})^n
\quad\text{for }n\ge n_1,
$$
where $c=(13/5)^{-n_1}\log\log|\gamma_{n_1}|>0$. Hence,
$$
n<\log\log\log|\gamma_n|\le\log{c}+n\log\tfrac{13}{5}.
$$
Since $\log\frac{13}{5}<1$, this is false for sufficiently large $n$,
proving the claim.

Choose a large $m\ge n_0$ satisfying \eqref{eq:biggap}, and set
$\delta_m=(\log|\gamma_m|)^{-\frac{12}{5}}$.
Then using the trivial bound $e^{e^e}\le|\gamma_n|\le|\gamma_m|$ for
$n_0\le n\le m$, we have
$$
\sum_{n=n_0}^m|\gamma_n|^{-\delta_m}\log^2|\gamma_n|
\le m\log^2|\gamma_m|<(\log\log\log|\gamma_m|)\log^2|\gamma_m|
\le(\log|\gamma_m|)^{\frac{11}{5}}
=\delta_m^{-\frac{11}{12}},
$$
since $\log\log{x}\le x^{\frac15}$ for all $x>1$.

To estimate the contribution from $n>m$ we apply integration
by parts.  Set $g(t)=t^{-\delta_m}\log^2{t}$. Then $g'(t)<0$ for
$t>e^{2/\delta_m}=\exp(2(\log|\gamma_m|)^{12/5})$; in particular,
if $m$ is sufficiently large then, by \eqref{eq:biggap}, $g'(t)<0$
for $t\ge|\gamma_{m+1}|$. Hence,
we have
\begin{align*}
\sum_{n=m+1}^\infty g(|\gamma_n|)
&=\lim_{\varepsilon\to0^+}\int_{|\gamma_{m+1}|-\varepsilon}^\infty g(t)\,dN(t)
=\int_{|\gamma_{m+1}|}^\infty(-g'(t))(N(t)-m)\,dt\\
&\le\int_{|\gamma_{m+1}|}^\infty(-g'(t))(\log\log\log{t})\,dt
\le\delta_m\int_{|\gamma_{m+1}|}^\infty
t^{-\delta_m-1}(\log{t})^{\frac{11}{5}}\,dt\\
&=\delta_m\int_{\log|\gamma_{m+1}|}^\infty e^{-\delta_mu}u^{\frac{11}5}\,du.
\end{align*}
Applying integration by parts three times
and using that $\delta_m\log|\gamma_{m+1}|\gg1$, we get
$$
\delta_m\int_{\log|\gamma_{m+1}|}^\infty e^{-\delta_mu}u^{\frac{11}5}\,du
\ll|\gamma_{m+1}|^{-\delta_m}(\log|\gamma_{m+1}|)^{\frac{11}{5}}.
$$
Note that $\delta_m=(\log|\gamma_m|)^{-\frac{12}{5}}
\ge(\log|\gamma_{m+1}|)^{-\frac{12}{13}}$, so
$|\gamma_{m+1}|^{-\delta_m}\le\exp(-(\log|\gamma_{m+1}|)^{\frac1{13}})$.
Hence, we conclude that
$$
\sum_{n=m+1}^\infty g(|\gamma_n|)\ll
\exp(-(\log|\gamma_{m+1}|)^{\frac1{13}})(\log|\gamma_{m+1}|)^{\frac{11}{5}}
\ll 1.
$$

Thus, altogether we have
$$
\delta_m^{-1}\ll1+\sum_{n=n_0}^\infty|\gamma_n|^{-\delta_m}\log^2|\gamma_n|
\ll1+\delta_m^{-\frac{11}{12}}.
$$
This is false for sufficiently large $m$, so \eqref{eq:omega3} must hold
for some arbitrarily large $T$.
\end{proof}

\section{Proofs of Theorems~\ref{thm:twist} and \ref{thm:oddN}}
We begin with an overview of the argument. By
Proposition~\ref{NfaqTlowerbound}, $\Ns_{f,a,q}(T)$ is sometimes large
if there exists $\alpha\in\Q^\times$ for which $H_{f,a,q,\alpha}(s)$
has a pole with large real part. The main obstacle to showing this is
that $H_{f,a,q,\alpha}(s)$ is defined as the difference of two functions
(cf.~\eqref{eq:Hdef}), whose poles could in principle cancel out. However,
as we show, there are some dependencies between $H_{f,a,q,\alpha}(s)$
for various choices of $(a,q,\alpha)$, from which it follows that
there is a suitable pole for at least one choice of inputs. More
specifically, in Lemma~\ref{lem:holo} we exhibit a relationship between
$H_{f,1,1,a/p}(s)$ and $H_{f,a,q,-a/q}(s)$, where $p$ and $q$ are primes
satisfying $pq\equiv-1\pmod*{Na}$. For any prime $p\nmid N$, we show that
there is some choice of $a\in\Z$ for which this leads to poles at the
simple zeros of $\Lambda_f(s)$, and thanks to \cite[Theorem~1.1]{Boo16},
those exist in abundance. Ultimately this implies that at least one of
$\Ns_f(T)$, $\Ns_{f,a,p}(T)$, $\Ns_{f,a,q}(T)$ is large, which yields
Theorem~\ref{thm:twist}.  Choosing $p=2$ and appealing to the second
and third conclusions of Proposition~\ref{NfaqTlowerbound} yields
Theorem~\ref{thm:oddN}.

Proceeding, given a prime $p$ and $a\in\Z$ coprime to $p$, define
\begin{equation}\label{eq:Cdef}
C_{f,a,p}(s)=\Delta_{f,a,p}(s)-\xi(p)p^{1-2s}\Delta_f(s).
\end{equation}
\begin{lemma}\label{lem:holo}
Let $a\in\Z$, and let $p$ and $q$ be prime numbers such that
$pq\equiv-1\pmod*{Na}$. Then
\begin{enumerate}
\item[(i)] $C_{f,a,p}(s)-\bigl(H_{f,1,1,a/p}(s)-\xi(p)p^{1-2s}H_{f,a,q,-a/q}(s)\bigr)$ is holomorphic for $\Re(s)>0$;
\item[(ii)] $\displaystyle{\sum_{b=1}^{p-1}C_{f,b,p}(s)}=-P_{f,p}(p^{1-s})\Delta_f(s)$.
\end{enumerate}
\end{lemma}
\begin{proof}
We first consider $H_{f,a,q,\alpha}(s)$, where $\alpha=-a/q$.
We have
$$
\Delta_{f,a,q}(s,\alpha)-\Delta_f(s)
=-R_{f,q}(q^{-s})\Lambda_f(s),
$$
which is holomorphic for $\Re(s)>0$.
Set $a'=-\frac{1+pq}{Na}$, so that
$\frac{a'}{q}-\frac1{Nq^2\alpha}=-\frac{p}{Na}$.
Let $r_{\bar{f},q}(j)$ be the numbers such that
$$
R_{\bar{f},q}(x)=\sum_{j=1}^\infty r_{\bar{f},q}(j)x^j.
$$
By Fourier inversion, we have
$$
\sum_{\substack{j\ge1\\j\equiv{t}\;(\text{mod }\varphi(Na))}}
r_{\bar{f},q}(j)x^j
=\frac1{\varphi(Na)}\sum_{\ell=1}^{\varphi(Na)}
e\!\left(-\frac{\ell{t}}{\varphi(Na)}\right)
R_{\bar{f},q}\!\left(e\!\left(\frac\ell{\varphi(Na)}\right)x\right).
$$
Thus,
\begin{align*}
\Delta_{\bar{f},a',q}&\!\left(s,-\frac1{Nq^2\alpha}\right)
-\Delta_{\bar{f}}\!\left(s,-\frac{p}{Na}\right)
=-\sum_{j=1}^\infty r_{\bar{f},q}(j)q^{-js}
\Lambda_{\bar{f}}\!\left(s,\frac{q^{j-1}}{Na}\right)\\
&=-\sum_{t=1}^{\varphi(Na)}\Lambda_{\bar{f}}\!\left(s,\frac{q^{t-1}}{Na}\right)
\sum_{\substack{j\ge1\\j\equiv{t}\;(\text{mod }\varphi(Na))}}
r_{\bar{f},q}(j)q^{-js}\\
&=-\frac1{\varphi(Na)}
\sum_{t=1}^{\varphi(Na)}\Lambda_{\bar{f}}\!\left(s,\frac{q^{t-1}}{Na}\right)
\sum_{\ell=1}^{\varphi(Na)}
e\!\left(-\frac{\ell{t}}{\varphi(Na)}\right)
R_{\bar{f},q}\!\left(e\!\left(\frac\ell{\varphi(Na)}\right)q^{-s}\right),
\end{align*}
which is again holomorphic for $\Re(s)>0$.
Hence, up to a holomorphic function, $H_{f,a,q,\alpha}(s)$ is
$$
\Delta_f(s)-\epsilon\xi(q)(-i\sgn{a})^k(Na^2)^{s-\frac12}
\Delta_{\bar{f}}\!\left(s,-\frac{p}{Na}\right).
$$

Next note that
$$
C_{f,a,p}(s)-H_{f,1,1,a/p}(s)=
\epsilon(i\sgn{a})^k\left(\frac{Na^2}{p^2}\right)^{s-\frac12}
\Delta_{\bar{f}}\!\left(s,-\frac{p}{Na}\right)
-\xi(p)p^{1-2s}\Delta_f(s)-R_{f,q}(q^{-s})\Lambda_f(s).
$$
Therefore, since $\xi(p)\xi(q)=\xi(-1)=(-1)^k$, we see that
$$
C_{f,a,p}(s)-H_{f,1,1,a/p}(s)+\xi(p)p^{1-2s}H_{f,a,q,\alpha}(s)
$$
is holomorphic for $\Re(s)>0$.

Finally, by \eqref{Deltafaq} we have
\begin{align*}
\sum_{b=1}^{p-1}C_{f,b,p}(s)
&=(p-1)\left[1-\frac{p}{p-1}P_{f,p}(p^{-s})-\xi(p)p^{1-2s}\right]\Delta_f(s)\\
&=-P_{f,p}(p^{1-s})\Delta_f(s).
\end{align*}
\end{proof}

In the following we shall make frequent use of the observation that for
any pair $h_1,h_2$ of meromorphic functions,
\begin{equation}\label{eq:thetah1h2}
\Theta(h_1+h_2)\le\max(\Theta(h_1),\Theta(h_2)),
\quad\text{with equality when }\Theta(h_1)\ne\Theta(h_2).
\end{equation}

Fix a prime $p\nmid N$.
By \cite[Theorem~1.1]{Boo16} and the functional equation,
$\Delta_f(s)$ has a pole with real part $\ge\frac12$, and thus
\begin{equation}\label{thetaf11lb}
\Theta(\Delta_f)=\theta_{f,1,1}\ge\frac12.
\end{equation}
Since all zeros of $P_{f,p}(p^{1-s})$ have real part $1$,
this is also true of
$P_{f,p}(p^{1-s})\Delta_f(s)$. Hence, by
Lemma~\ref{lem:holo}(ii), there exists
$a\in\{1,\ldots,p-1\}$ such that
$C_{f,a,p}(s)$ has a pole with real part $\ge\frac12$ and
satisfies $\Theta(C_{f,a,p})\ge\theta_{f,1,1}$.
By \eqref{eq:Cdef} and \eqref{eq:thetah1h2}, it follows that
\begin{equation}\label{thetaCfap}
\Theta(C_{f,a,p})=\max(\Theta(\Delta_{f,a,p}),\theta_{f,1,1}).
\end{equation}
Let $q$ be a prime satisfying $pq\equiv-1\pmod*{Na}$, and set
$a'=-(1+pq)/(Na)$.

We aim to prove that
\begin{equation}\label{eq:summary}
\max\bigl(\Ns_f(T),\Ns_{f,a,p}(T),\Ns_{f,a,q}(T)\bigr)
=\Omega\bigl(T^{\frac16-\varepsilon}\bigr)
\quad\text{for all }\varepsilon>0.
\end{equation}
To that end, we will show that at least one of the following
inequalities holds for some $\alpha\in\Q^\times$:
\begin{itemize}
\item[(i)]
$\max(\Theta(H_{f,1,1,\alpha}),\Theta(H_{\bar{f},1,1,\alpha}))
\ge\theta_{f,1,1}\ge\frac12$;
\item[(ii)]
$\max(\Theta(H_{f,a,p,\alpha}),\Theta(H_{\bar{f},a',p,\alpha}))
\ge\theta_{f,a,p}\ge\frac12$;
\item[(iii)]
$\max(\Theta(H_{f,a,q,\alpha}),\Theta(H_{\bar{f},a',q,\alpha}))
\ge\theta_{f,a,q}\ge\frac12$.
\end{itemize}
To see that this suffices, suppose for instance that (iii) holds.
By Proposition~\ref{voronoi}, we have $\Ns_{f,a,q}(T)=\Ns_{\bar{f},a',q}(T)$
and $\theta_{f,a,q}=\theta_{\bar{f},a',q}$.
Thus, applying Proposition~\ref{NfaqTlowerbound} to either $(f,a,q)$ or
$(\bar{f},a',q)$, we conclude that
$$
\Ns_{f,a,q}(T)=\Omega(T^{\beta-\varepsilon}),
\quad\text{where }
\beta\ge\frac13(1-\theta_{f,a,q})+\theta_{f,a,q}-\frac12
=\frac{2\theta_{f,a,q}}{3}-\frac16\ge\frac16.
$$
If, instead, (i) or (ii) holds, then by a similar argument we find that
$\Ns_{f,1,1}(T)=\Omega(T^{\beta-\varepsilon})$
or $\Ns_{f,a,p}(T)=\Omega(T^{\beta-\varepsilon})$
for some $\beta\ge\frac16$.
Hence, \eqref{eq:summary} follows in any case.

Let us suppose that conditions (i) and (iii) are false for all
$\alpha\in\Q^\times$ and show that this leads to (ii).
Since (i) is false, in view of \eqref{thetaf11lb} we must have
$\theta_{f,1,1}>\Theta(H_{f,1,1,a/p})$.  In turn, by \eqref{thetaCfap}
this implies that $\Theta(C_{f,a,p})>\Theta(H_{f,1,1,a/p})$.  Hence,
by Lemma~\ref{lem:holo}(i) and \eqref{eq:thetah1h2}, we have
$\Theta(H_{f,a,q,-a/q})=\Theta(C_{f,a,p})$. By \eqref{thetaCfap}, this
implies $\Theta(H_{f,a,q,-a/q})\ge\theta_{f,1,1}>0$,
and thus $\theta_{f,a,q}\ge\frac12$, by
Proposition~\ref{NfaqTlowerbound}.

Next, by \eqref{thetafaq} we have
$\theta_{f,a,p}=\max(\Theta(\Delta_{f,a,p}),\Theta(\Delta_{\bar{f},a',p}))$.
If
\begin{equation}\label{cond1}
\Theta(\Delta_{\bar{f},a',p})
\le\max(\Theta(\Delta_{f,a,p}),\theta_{f,1,1})
=\Theta(H_{f,a,q,-a/q})
\end{equation}
then it follows that
\begin{equation}\label{ThetaHf}
\Theta(H_{f,a,q,-a/q})=\max(\theta_{f,a,p},\theta_{f,1,1}).
\end{equation}

Suppose now that \eqref{cond1} is false. Then
$\Theta(\Delta_{\bar{f},a',p})
>\max(\Theta(\Delta_{f,a,p}),\theta_{f,1,1})$, so that
$$
\theta_{f,a,p}=\Theta(\Delta_{\bar{f},a',p})>\theta_{f,1,1}.
$$
Since (i) is false, this implies that
$\Theta(\Delta_{\bar{f},a',p})
>\max(\Theta(H_{\bar{f},1,1,a'/p}),\theta_{f,1,1})$. By
\eqref{eq:Cdef} and Lemma~\ref{lem:holo}(i) with $(\bar{f},a')$
in place of $(f,a)$, it follows from \eqref{eq:thetah1h2} that
\begin{equation}\label{ThetaHbarf}
\Theta(H_{\bar{f},a',q,-a'/q})=\Theta(\Delta_{\bar{f},a',p})
=\max(\theta_{f,a,p},\theta_{f,1,1}).
\end{equation}

Therefore,
since at least one of \eqref{ThetaHf} and \eqref{ThetaHbarf} must hold,
we have
$$
\max(\Theta(H_{f,a,q,-a/q}),\Theta(H_{\bar{f},a',q,-a'/q}))
\ge\max(\theta_{f,a,p},\theta_{f,1,1}).
$$
Since (iii) is false, this implies that
$\theta_{f,a,q}>\max(\theta_{f,a,p},\theta_{f,1,1})$.
Hence, by \eqref{thetafaq}, either
\begin{equation}\label{branch}
\Theta(\Delta_{f,a,q})
>\max(\theta_{f,a,p},\theta_{f,1,1})
\quad\text{or}\quad
\Theta(\Delta_{\bar{f},a',q})
>\max(\theta_{f,a,p},\theta_{f,1,1}).
\end{equation}

Suppose that the first inequality in \eqref{branch} holds. Then
by \eqref{eq:Cdef}
(with $q$ in place of $p$) and \eqref{eq:thetah1h2}, we
have $\Theta(C_{f,a,q})=\Theta(\Delta_{f,a,q})>\theta_{f,1,1}$.
Since (i) is false, this implies
$\Theta(C_{f,a,q})>\Theta(H_{f,1,1,a/q})$.
On the other hand, by Lemma~\ref{lem:holo}(i)
(with the roles of $p$ and $q$ reversed)
and \eqref{eq:thetah1h2}, we have
$$\Theta(H_{f,a,p,-a/p})=\Theta(C_{f,a,q})
=\Theta(\Delta_{f,a,q})>\theta_{f,a,p}.$$
This also implies that $\Theta(H_{f,a,p,-a/p})>0$, whence
$\theta_{f,a,p}\ge\frac12$, by Proposition~\ref{NfaqTlowerbound}.

If, instead, the second inequality holds in \eqref{branch},
then running through the same argument with $(\bar{f},a')$
in place of $(f,a)$, we find that
$$\Theta(H_{\bar{f},a',p,-a'/p})=\Theta(C_{\bar{f},a',q})
=\Theta(\Delta_{\bar{f},a',q})>\theta_{\bar{f},a',p}\ge\frac12.
$$
Hence, in either case we see that (ii) holds, and this
concludes the proof of \eqref{eq:summary}.

Now, by \eqref{eq:summary} and \eqref{Deltafaq}, it follows that
there is a character $\chi$ of conductor $1$, $p$ or $q$ such
that $\Ns_{f\otimes\chi}(T)=\Omega(T^{\frac16-\varepsilon})$
for all $\varepsilon>0$. This implies Theorem~\ref{thm:twist}.

For the proof of Theorem~\ref{thm:oddN}, we may assume that
$\Ns_f(T)\ll1+T^\varepsilon$ for all $\varepsilon>0$, since
the result is trivial otherwise. To avoid contradicting
Proposition~\ref{NfaqTlowerbound},
it must therefore be the case that
$\max(\Theta(H_{f,1,1,\alpha}),\Theta(H_{\bar{f},1,1,\alpha}))
\le\frac12$ for all $\alpha\in\Q^\times$.

Since $N$ is odd, we can take $p=2$ and $a=1$ in the above, and choose
any suitable prime $q$. Then by Lemma~\ref{lem:holo}(ii), we have
$$
\Delta_{f,a,p}(s)=\bigl(\xi(p)p^{1-2s}-P_{f,p}(p^{1-s})\bigr)\Delta_f(s),
$$
and it follows that $\Ns_{f,a,p}(T)\le\Ns_f(T)$ and
$\max(\Theta(H_{f,a,p,\alpha}),\Theta(H_{\bar{f},a',p,\alpha}))
\le\frac12$ for all $\alpha\in\Q^\times$.  Thus, by \eqref{eq:summary},
$\Ns_{f,a,q}(T)=\Omega(T^{\frac16-\varepsilon})$ for all
$\varepsilon>0$.  Therefore, by Proposition~\ref{voronoi}, at least one
of $\Delta_{f,a,q}(s),\Delta_{\bar{f},a',q}(s)$ has a pole in the region
$\{s\in\C:\Re(s)\ge\frac12\}$ that is not a pole of $\Delta_f(s)$. By
\eqref{eq:Cdef} and Lemma~\ref{lem:holo}(i), the same applies to one
of $H_{f,1,1,a/q}(s)$, $H_{f,a,p,-a/p}(s)$, $H_{\bar{f},1,1,a'/q}(s)$,
or $H_{\bar{f},a',p,-a'/p}(s)$.

Since
$$
\Ns_{f,a,p}(T)=\Ns_{\bar{f},a',p}(T)\le
\Ns_{f,1,1}(T)=\Ns_{\bar{f},1,1}(T),
$$
whichever function has the pole, we can apply Proposition~\ref{NfaqTlowerbound}
to see that $\Ns_f(T)$ satisfies the second and third conclusions.
In particular, $\Ns_f(T)\ge\log\log\log{T}$
for some arbitrarily large $T$, and
if $k=1$ or $f$ is a CM form then Coleman's theorem \cite{Col90}
implies that
$$
1-\theta_{f,a,p}(T)\ge
1-\theta_{f,1,1}(T)\gg
(\log{T})^{-\frac23}(\log\log{T})^{-\frac13}
\quad\text{for all }T\ge3,
$$
whence $\Ns_f(T)=\Omega(\exp((\log{T})^{\frac13-\varepsilon}))$
for all $\varepsilon>0$.
Moreover, since $\Ns_f(T)\ll1+T^\varepsilon$, we must have
$\theta_{f,1,1}=1$, so $\Lambda_f(s)$ has simple zeros with real part
arbitrarily close to $1$.

Finally, by Lemma~\ref{lem:holo}(ii)
we have $\Theta(C_{f,a,p})=1$.
Since $\Theta(H_{f,1,1,a/p})\le\frac12$,
Lemma~\ref{lem:holo}(i) and \eqref{eq:thetah1h2} imply that
$\Theta(H_{f,a,q,-a/q})=1$.
Applying Proposition~\ref{NfaqTlowerbound}, it follows that
$\Ns_{f,a,q}(T)=\Omega(T^{\frac12-\varepsilon})$ for all
$\varepsilon>0$.
By Lemma~\ref{lem:holo}(i), $C_{f,a,q}(s)$ and $C_{\bar{f},a',q}(s)$ are
holomorphic for $\Re(s)>\frac12$.
Hence, by \eqref{eq:Cdef} and Proposition~\ref{voronoi}, all poles of
$\Delta_{f,a,q}^*(s)$ that are not poles of $\Delta_f^*(s)$
lie on the line $\{s\in\C:\Re(s)=\frac12\}$.
Since $\Ns_f(T)\ll1+T^\varepsilon$, $\Delta_{f,a,q}^*(s)$ must have
$\Omega(T^{\frac12-\varepsilon})$ poles with real part $\frac12$ and
imaginary part in $[-T,T]$. By \eqref{Deltafaq}, the same applies to
$\Delta_f(s,\chi)$ for some $\chi\pmod*{q}$.

\bibliographystyle{amsalpha}
\bibliography{simplezeros}
\end{document}